\documentclass{article}

\usepackage{subfigure}
\usepackage{tikz}

\usetikzlibrary{graphs,graphs.standard,calc}
\usepackage[utf8]{inputenc}

\usepackage[utf8]{inputenc}
\usepackage[english]{babel}
\usepackage{amsmath}
\usepackage{times}
\usepackage{graphicx}
\usepackage{float}
\usepackage[margin=1in]{geometry}
\usepackage{blindtext}
\usepackage{amssymb}
\usepackage{tikz}
\usepackage{amsthm}
\usepackage{mathtools}
\usepackage{textcmds}
\DeclarePairedDelimiter{\floor}{\lfloor}{\rfloor}
\DeclarePairedDelimiter{\ceil}{\lceil}{\rceil}
\usepackage{comment}
\newtheorem{theorem}{Theorem}[section]

\newtheorem{proposition}[theorem]{Proposition}
\newtheorem{lemma}[theorem]{Lemma}
\newtheorem{conjecture}[theorem]{Conjecture}
\theoremstyle{remark}
\newtheorem{remark}{Remark}
\newtheorem{case}{Case}
\newtheorem{claim}{Claim}
\theoremstyle{definition}

\title{The spectral radius of graphs with no odd wheels}
\date{\today}
\author{Sebastian Cioab\u{a}\thanks{Department of Mathematical Sciences, University of Delaware, \texttt{cioaba@udel.edu}. Research partially supported by National Science Foundation grant CIF-1815922.} \and Dheer Noal Desai\thanks{Department of Mathematical Sciences, University of Delaware, \texttt{dheernsd@udel.edu}} \and Michael Tait\thanks{Department of Mathematics \& Statistics, Villanova University, \texttt{michael.tait@villanova.edu}. Research partially supported by National Science Foundation grant DMS-2011553.}}
\begin{document}

\maketitle

\begin{abstract}
    The odd wheel $W_{2k+1}$ is the graph formed by joining a vertex to a cycle of length $2k$. In this paper, we investigate the largest value of the spectral radius of the adjacency matrix of an $n$-vertex graph that does not contain $W_{2k+1}$. We determine the structure of the spectral extremal graphs for all $k\geq 2, k\not\in \{4,5\}$. When $k=2$, we show that these spectral extremal graphs are among the Tur\'{a}n-extremal graphs on $n$ vertices that do not contain $W_{2k+1}$ and have the maximum number of edges, but when $k\geq 9$, we show that the family of spectral extremal graphs and the family of Tur\'{a}n-extremal graphs are disjoint.
\end{abstract}
\section{Introduction}

Given a graph $F$, the {\em Tur\'an number of $F$} is denoted by $\mathrm{ex}(n, F)$ and is the maximum number of edges in an $n$-vertex graph which does not contain $F$ as a subgraph. Many questions in extremal combinatorics can be rephrased as asking for a certain Tur\'an number, and hence the study of the function $\mathrm{ex}(n, F)$ for various $F$ (or, more generally, for various families of forbidden graphs) is one of the most important topics in graph theory and combinatorics. The area has been studied extensively since its introduction to the present (see surveys \cite{degeneratesurvey, keevash, sidorenko}). The family of $F$-free graphs with $\mathrm{ex}(n, F)$ edges is denoted by $\mathrm{EX}(n, F)$. 

Given a graph $G$, let $\lambda_1(G)$ be the largest eigenvalue of its adjacency matrix (also called the {\em spectral radius of $G$}). We study a spectral version of the Tur\'an problem: given a graph $F$, let $\mathrm{spex}(n, F)$ denote the maximum value of $\lambda_1(G)$ over all $n$-vertex graphs $G$ which do not contain $F$ as a subgraph. We denote the family of $F$-free graphs with spectral radius equal to $\mathrm{spex}(n, F)$ by $\mathrm{SPEX}(n, F)$. The study of $\mathrm{spex}(n, F)$ for various graphs $F$ (or various families of forbidden subgraphs) 
was first proposed in generality by Nikiforov \cite{Nikiforovpaths}, though several sporadic results appeared earlier. In particular, the maximum spectral radius of graphs with no $K_{r+1}$ was determined in \cite{NK} and the maximum spectral radius of graphs with no $K_{s,t}$ was upper bounded in \cite{BG} and \cite{NZ}. 

One motivation for studying such problems is that $\lambda_1(G)$ is an upper bound for the average degree of $G$, and hence any upper bound on $\mathrm{spex}(n, F)$ also gives an upper bound on $\mathrm{ex}(n, F)$. Indeed the results in \cite{NK, NZ} imply Tur\'an's theorem as well as  F\"uredi's improvement to the K\H{o}v\'ari-S\'os-Tur\'an theorem \cite{F2}. The function $\mathrm{spex}(n, F)$ has been studied for many families of graphs  (see, for example, \cite{FG, LLT, NO, SSbook, W, YWZ, ZW, ZWF}). The study of $\mathrm{spex}(n, F)$ fits into a a broader framework of {\em Brualdi-Solheid problems} \cite{BS} investigating the maximum spectral radius over all graphs belonging to a specified family. Many results are known in this area (for example \cite{BZ, BLL, EZ, FN, S, SAH}). 

Let $W_t$ be the {\em wheel graph} on $t$ vertices: the graph formed by joining a vertex to all of the vertices in a cycle on $t-1$ vertices. In \cite{ZHL}, the authors determine the maximum spectral radius over $n$-vertex graphs which forbid all wheels. They state that ``it seems difficult to determine the maximum spectral radius of a $\{W_l\}$-free graph of order $n$". In this paper, we study $\mathrm{spex}(n, W_{2k+1})$ and $\mathrm{SPEX}(n,W_{2k+1})$ and answer their question except when $k \in \{4,5\}$ as described below.
The Tur\'an problem for odd wheels was recently resolved in \cite{DT} and \cite{Y} and we record these results here to compare them with ours. 

\begin{theorem}[Dzido and Jastrz\c ebski \cite{DT}]\label{thmDJ}
Let $W_5$ be the wheel on $5$ vertices. Then \[\textup{ex}(n, W_5) = \begin{cases}\frac{n^2}{4} + \frac{n}{2} - 1 & n\equiv{2}\pmod{4} \\ \left\lfloor \frac{n^2}{4} \right\rfloor + \left\lfloor\frac{n}{2} \right\rfloor & \mathrm{otherwise}.

\end{cases}\]

\end{theorem}

The extremal graphs in $\mathrm{EX}(n, W_5)$ consist of a complete bipartite graph with an additional maximum matching in each part. When $n\not\equiv 2 \pmod{4}$, the bipartite graph is as balanced as possible. When $n\equiv 2 \pmod{4}$ there are two extremal graphs: choosing the bipartite graph to have parts of size $n/2-1$ and $n/2+1$ gives the same number of edges as choosing them to each be of size $n/2$. Our first theorem shows that the spectral extremal graphs are a subset of the Tur\'an-extremal graphs.

\begin{theorem}
\label{thm2}For sufficiently large $n$, $\mathrm{SPEX}(n, W_5) \subset \mathrm{EX}(n, W_5). $
\end{theorem}

When $k\geq 3$, the structure of the extremal graphs changes. 

\begin{theorem}[Yuan \cite{Y}]
\label{thm3}
Let $k\geq 3$ be an integer. For $n$ sufficiently large, \[\textup{ex}(n, W_{2k+1}) = \max\bigg\{n_0 n_1 + \bigg\lfloor \frac{(k-1)n_0}{2}\bigg\rfloor + 1 : n_0 + n_1 = n\bigg\}. \]

\end{theorem}

A graph is called {\em nearly $(k-1)$-regular} if every vertex but one has degree $k-1$ and the final vertex has degree $k-2$. Let $\mathcal{U}_{k,n}$ be the family of $(k-1)$-regular or nearly $(k-1)$-regular graphs on $n$ vertices which do not contain a path on $2k-1$ vertices. This family is non-empty when $n\geq 2k$ (see Proposition 2.1 in \cite{Y}). In \cite{Y} it is shown that for $k\geq 3$, the family $\mathrm{EX}(n, W_{2k+1})$ consists of complete bipartite graphs with parts of size $\frac{n}{2} - r$ and $\frac{n}{2}+r$ along with a graph from $\mathcal{U}_{k, n/2+r}$ embedded in the larger part and a single edge embedded in the smaller part, where $\frac{n}{2}+r \in \Bigg\{\bigg\lfloor\frac{2n + k - 1}{4}\bigg\rfloor, \bigg\lceil\frac{2n + k - 1}{4}\bigg\rceil\Bigg\}$. Our second theorem determines the structure of the graphs in $\mathrm{SPEX}(n,W_{2k+1})$ for $k\geq 3, k\not\in \{4,5\}$ and $n$ sufficiently large.

\begin{theorem}
\label{thm4}
Let $k\geq 3$, $k\not\in \{4,5\}$. For sufficiently large $n$, if $G \in \textup{SPEX}(n, W_{2k+1})$, then $G$ is the union of a complete bipartite graph with parts $L$ and $R$ of size $\frac{n}{2}+s$ and $\frac{n}{2}-s$, respectively,
and a graph from $\mathcal{U}_{k, n/2+s}$ embedded in $G[L]$ and exactly one edge in $G[R]$. Furthermore, $|s| \leq 1$.
\end{theorem}

As in Theorems \ref{thmDJ} and \ref{thm2}, the exact spectral extremal graph depends on the parity $n$ mod $4$. Furthermore, in the case where $|L|(k-1)$ is odd, it may depend on which particular graph in $\mathcal{U}_{k,n/2+s}$ is embedded in $L$, making it complicated to determine $\mathrm{SPEX}(n, W_{2k+1})$ precisely. Note though that Theorem \ref{thm4} implies that $\mathrm{SPEX}(n,W_{2k+1})\cap \mathrm{EX}(n, W_{2k+1})=\emptyset$ when $k=7$ or $k\geq 9$ and $n$ is sufficiently large. In Section \ref{section partition sizes} we give more information about the value of $s$ for the spectral extremal graph, and in most cases this allows one to determine $\mathrm{spex}(n, W_{2k+1})$ precisely. In particular, we can determine the exact value of $s$ if $n\not\equiv 2 \pmod{4}$ or if $k$ is odd, and in these cases we have that $\mathrm{spex}(n, W_{2k+1})$ is the root of a cubic polynomial. When $n\equiv 2 \pmod{4}$ and $k$ is even, we determine $\mathrm{spex}(n, W_{2k+1})$ up to an additive factor of $o(1/n)$.

Before beginning, we remark that both problems for even wheels are solved in a more general setting. Let $F$ be any graph of chromatic number $r+1\geq 3$ which contains an edge such that $\chi(F\setminus e) = r$. Simonovits \cite{simonovits} proved that for $n$ large enough, the only graph in $\mathrm{EX}(n, F)$ is the Tur\'an graph with $r$ parts. Nikiforov \cite{Nsat} proved a spectral version of this theorem. Since $\chi(W_{2k})=4$ and $\chi(W_{2k}\setminus e)=3$ for any edge $e$ of $W_{2k}$, these theorems apply. 

\subsection*{Organization and notation}
In Section \ref{section lemmas} we present several lemmas that we will use throughout the paper. In Section \ref{structural lemmas section}, we give structural results that graphs in $\mathrm{SPEX}(n, W_{2k+1})$ must satisfy for all $k$. We then specialize to the $k=2$ and $k>2$ cases. In Section \ref{section k2} we prove Theorem \ref{thm2} and in Section \ref{section k3} we prove Theorem \ref{thm4}. In Section \ref{section partition sizes}, we discuss the exact sizes of the partitions in the spectral extremal $W_{2k+1}$-free graphs. We end with some concluding remarks and open problems. 

For $G$ an $n$-vertex graph, the \textit{adjacency matrix} of $G$ is the $n$-dimensional $0,1$-matrix, $A(G) = (a_{ij})_{n\times n}$, with $a_{i,j} = 1$, if $v_i$ is adjacent to $v_j$, and $a_{ij} = 0$ otherwise. Therefore, $A(G)$ is a symmetric matrix and has $n$ real eigenvalues, that may be denoted in descending order as follows: $\lambda_1 \geq \lambda_2 \geq \ldots \geq \lambda_n$. We will also call $\lambda_1(G)$ the \textit{spectral radius} of $G$, and we will denote it by $\lambda_1$ if the associated graph is unambiguous by context. We will use $P_t, C_t, K_t$ and $W_t$ to denote the path, the cycle, complete graph, and the wheel on $t$ vertices, respectively.

\section{Background Lemmas}\label{section lemmas}
In this section, we record several lemmas that we will use. We start with the Triangle Removal Lemma and a stability theorem of F\"uredi.

\begin{lemma}[Triangle Removal Lemma \cite{ET}, \cite{FT}, \cite{RISE}]\label{triangle removal lemma}
For each $\epsilon > 0$, there exists a $\delta > 0$ and $N = N(\epsilon)$, such that every graph $G$ on $n$ vertices with $n \geq N$, and at most $\delta n^3$ triangles, can be made triangle-free by removing at most $\epsilon n^2$ edges.
\end{lemma}

\begin{lemma}[F\"uredi Stability Theorem \cite{F}]\label{furedi stability lemma}
Suppose $G$ is a triangle-free graph on $n$ vertices and $s$ is a positive integer such that $e(G) = e(T_{n, 2}) - s$. Then there exists a bipartite subgraph $H$, such that $e(H) \geq e(G) - s$.
\end{lemma}

Next, we will need the even-circuit theorem. We note that the best current bounds for $\mathrm{ex}(n, C_{2k})$ are given by He \cite{He} (see also Bukh and Jiang \cite{B}), but for our purposes the dependence of the multiplicative constant on $k$ is not important. We use the following version because it makes the calculations slightly easier.

\begin{lemma}[Even Circuit Theorem \cite{V}]
\label{lemC2k}
For $k\geq 2$ and $n$ a natural number, 
\[\textup{ex}(n, C_{2k}) \leq 8(k-1)n^{1+ 1/k}. \]
\end{lemma}

\begin{lemma}\label{inclusion exclusion lemma}
If $A_1, \ldots, A_q$ are finite sets, then 
\[|A_1 \cap \ldots \cap A_q| \geq \sum_{i=1}^q |A_i| - (q-1)\bigg| \bigcup_{i=1}^q A_i \bigg|.\]
\end{lemma}

\begin{lemma}[\cite{T} Theorem 4.4]
\label{thm5}
Let $H_1$ be a graph  on $n_0$ vertices with maximum degree $d$ and $H_2$ be a graph on $n-n_0$ vertices with maximum degree $d'$. $H_1$ and $H_2$ may have loops or multiple edges, where loops add $1$ to the degree. Let $H$ be the join of $H_1$ and $H_2$. Define \begin{equation}
B =
\begin{bmatrix} 
d & n-n_0\\
n_0 & d'\\
\end{bmatrix}.
\end{equation}
Then $\lambda_1(H) \leq \lambda_1(B)$.
\end{lemma}

\begin{lemma}[\cite{FG} Lemma 7]
\label{cor1}
If $G$ has $n$ vertices, $t$ triangles and spectral radius $\lambda_1 > \frac{n}{2}$, then 
$e(G) > \lambda_1^2 - \frac{6t}{n}$.
\end{lemma}

 We end this section with two straightforward but useful remarks. 
\begin{remark}
By Theorem \ref{thm3} and some straightforward computation, for fixed $k\geq 3$ and $n$ sufficiently large we have  

\begin{equation}\label{turan number of wheel bound}
    \mathrm{ex}(n, W_{2k+1}) \leq \frac{n^2}{4} + \frac{n(k-1)}{4} + \frac{(k-1)^2}{16} + 1 < \frac{n^2}{4} + \frac{nk}{4}.
\end{equation}
\end{remark}
\begin{remark}
A graph is $W_{2k+1}$-free if and only if the subgraph induced by the neighborhood of each vertex is $C_{2k}$-free.  
\end{remark}

\section{Structural results for extremal graphs}\label{structural lemmas section}

In this section, we will assume that $k\geq 2$ is fixed and that $G\in \textup{SPEX}(n, W_{2k+1})$. We will also use auxiliary constants $\epsilon$ and $p$ and $\theta$ and we will frequently assume that $n$ is larger than some constant depending only on $k, \epsilon, p, \theta$. Every lemma in this section holds only for $n$ large enough.

First, we will need a lower bound on $\mathrm{spex}(n, W_{2k+1})$. 
\begin{lemma}\label{lemma first lower bound quotient matrix}
Let $k\geq 2$ be an integer. 
Then,
$
\lambda_1(G)> 
\frac{n+k-1}{2}. 
$
\end{lemma}

\begin{proof}
For $k=2$, if $H\in \textup{EX}(n, W_{5})$ then since $G$ is a graph maximizing spectral radius over all $W_{5}$-free graphs, we have
\begin{equation}
\label{eqn1}
    \lambda_1(G) \geq \lambda_1(H) \geq \frac{\mathbf{1}^T A(H)\mathbf{1}}{\mathbf{1^T 1}} \geq  2\frac{\bigg(\frac{n^2}{4}+\frac{n}{2} - 1\bigg)}{n} = \frac{n}{2} + 1 - \frac{2}{n} > \frac{n+1}{2}.
\end{equation}

For $k \geq 3$, let 
\[
Q = \begin{bmatrix}
k-1 & \floor{\frac{n}{2}} \\
\ceil{\frac{n}{2}} & 0
\end{bmatrix},
\]
and let $\mu$ be the spectral radius of $Q$ with eigenvector $\begin{bmatrix} 1 & \eta \end{bmatrix}^T$. By direct computation, $\mu = \frac{k-1 + \sqrt{(k-1)^2 + 4 \ceil{\frac{n}{2}}\floor{\frac{n}{2}}}}{2} \geq \frac{k-1 + \sqrt{(k-1)^2 + n^2 - 1}}{2}$, and for any $\epsilon > 0$, we have that $|1-\eta| < \epsilon$ for $n$ large enough. Let $\mathbf{z}$ be the $n$-dimensional vector where the first $\lceil n/2 \rceil$ entries are $1$ and the last $\lfloor n/2 \rfloor$ entries are $\eta$. 

Now, let $G_1$ be a graph on $\ceil{\frac{n}{2}}$ vertices in $\mathrm{EX}(\ceil{\frac{n}{2}}, \{K_{1,k}, P_{2k-1}\})$. That is $G_1$ is a graph that is $P_{2k-1}$-free and is $(k-1)$-regular if $(k-1)\ceil{\frac{n}{2}}$ is even and is $(k-1)$-nearly regular otherwise. Define $\gamma = 0$ if $(k-1)\ceil{\frac{n}{2}}$ is even and $1$ if it is odd, and so $e(G_1) = (k-1)\ceil{\frac{n}{2}}/2 - \gamma/2$. Now let 
\[
\tilde{G} = G_1 + (K_2 \cup (\floor{\frac{n}{2}}-2)K_1).
\]

Let $A(\tilde{G})$ be indexed so that the vertices corresponding to $G_1$ are first, and let $A(E_1)$ be a diagonal matrix with exactly $\gamma$ entries equal to $1$ (and the rest $0$s) so that the principal submatrix of $A(\tilde{G}) + A(E_1)$ corresponding to the first $\ceil{\frac{n}{2}}$ vertices has constant row sum $(k-1)$.  Since $\tilde{G}$ is $W_{2k+1}$-free, we have that 
\begin{equation}\label{lower bound G}
\begin{split}
\lambda_1(G) &\geq \lambda_1(\tilde{G}) \geq \frac{ \mathbf{z}^TA(\tilde{G})\mathbf{z}}{\mathbf{z}^T\mathbf{z}} = \mu - \frac{ \mathbf{z}^TA(E_1)\mathbf{z}}{\mathbf{z}^T\mathbf{z}} + \frac{2\eta^2}{\mathbf{z}^T\mathbf{z}} > \mu - \frac{\gamma}{\mathbf{z}^T\mathbf{z}} + \frac{2(1-\epsilon)^2}{\mathbf{z}^T\mathbf{z}} \\ & > \frac{k-1 + \sqrt{(k-1)^2 + n^2 - 1}}{2} + \frac{2(1-\epsilon)^2 - \gamma}{n} > \frac{n+k-1}{2}.
\end{split}
\end{equation}
\end{proof}

Next we show that $G$ contains a large maximum cut.

\begin{lemma}\label{lemma first maxcut}
For any $\epsilon >0$, there is a partition $V(G) = S\sqcup T$ which forms a maximum cut satisfying 
\[
e(S,T) \geq \left(\frac{1}{4} - \epsilon\right)n^2.
\]
Furthermore, 
\[
\left(\frac{1}{2} - \sqrt{\epsilon}\right)n \leq |S|, |T| \leq \left(\frac{1}{2} + \sqrt{\epsilon}\right) n.\]
\end{lemma}
\begin{proof}
Fix $\epsilon > 0$, and let $\delta, N_1$ be the constants that come from the Triangle Removal Lemma (Lemma \ref{triangle removal lemma}) with constant $\epsilon/4$. That is, $\delta$ and $N_1$ are chosen so that any graph on $n\geq N_1$ vertices and at most $\delta n^3$ triangles can be made triangle-free by removing at most $\frac{\epsilon}{4} n^2$ edges. Now, since $G$ is $W_{2k+1}$-free, any neighborhood of a vertex does not contain a $C_{2k}$. Letting $t$ be the number of triangles in $G$, we have 
\begin{equation}\label{triangle upper bound equation}
3t = \sum_{v\in V(G)} e(G[N(v)]) \leq \sum_{v\in V(G)} \textup{ex}(d(v), C_{2k}) \leq \sum_{v\in V(G)}\textup{ex}(n, C_{2k}) \leq 8(k-1)n^{2+1/k}.
\end{equation}

Thus, for $n$ a large enough constant depending only on $k$ and $\delta$ (and hence only on $k$ and $\epsilon$), we have that 
\[
t\leq \frac{8}{3}(k-1)n^{2+1/k} < \delta n^3.
\]
By Lemma \ref{triangle removal lemma}, for $n\geq N_1$, there is a triangle-free subgraph $G_1$ obtained by deleting at most $\frac{\epsilon}{4}n^2$ edges from $G$. Since $G_1$ is triangle-free, we may define $s = e(T_{n,2}) - e(G_1) \geq 0$. By F\"uredi's stability theorem (Lemma \ref{furedi stability lemma}), $G_1$ contains a bipartite subgraph $G_2$ with at least $e(G_1) - s$ edges. We now have a bipartite subgraph $G_2$ of $G$ such that $e(G_2) \geq e(G) -\frac{\epsilon}{4}n^2-s$. 

To lower bound the number of edges in $G$, we use Lemma \ref{cor1}, Lemma \ref{lemma first lower bound quotient matrix}, and \eqref{triangle upper bound equation}  to get that
\begin{equation}\label{first edge lower bound equation}
e(G) > \lambda_1^2 - \frac{6t}{n} > \frac{n^2}{4} - 16(k-1)n^{1+1/k} > \frac{n^2}{4} - \frac{\epsilon}{4}n^2,
\end{equation}
for $n$ large enough. This implies that $e(G_1) \geq e(G) - \frac{\epsilon}{4}n^2 \geq \frac{n^2}{4} - \frac{\epsilon}{2}n^2$ and hence $s\leq \frac{\epsilon}{2}n^2$. Therefore we have a bipartite subgraph $G_2$ with partite sets $S$ and $T$ satisfying
\[
e(G_2) \geq e(G_1) - s \geq \frac{n^2}{4} - \epsilon n^2.
\]
The bounds on the sizes of $|S|$ and $|T|$ follow from $|S||T|\geq e(G_2)$ and the inequality above. 
\end{proof}
Next we will show that most vertices have degree close to $\frac{n}{2}$. Define
\[P:=\bigg\{
v:d(v) \leq \bigg( \frac{1}{2} - \frac{1}{p} \bigg)n
\bigg\}.\]
\begin{lemma}
\label{lem5.2}

Let $p$ be a fixed natural number. Then the set $P$ satisfies
\[|P| \leq 16pkn^{1/k}.\]
\end{lemma}

\begin{proof}

We will ignore floors and ceilings throughout this proof. Assume to the contrary that the set of `atypical' vertices, $P$, has cardinality greater than $16pkn^{1/k}$. Then consider any fixed subset $P' \subseteq P$ with $|P'| = 16p kn^{1/k}$.  Using \eqref{first edge lower bound equation}, it follows that

\begin{align*}
\label{eqn24}
&\textup{ex}(n-16pkn^{1/k}, W_{2k+1}) \geq e[G\setminus P'] \geq e(G) - \sum_{v \in P'}d(v) \geq \frac{n^2}{4} - 16(k-1) n^{1 + 1/k} - 16pk n^{1/k}\bigg( \frac{1}{2} - \frac{1}{p} \bigg)n \\
=& \left(\dfrac{(n - 16pkn^{1/k})^2}{4} + \frac{(n - 16pkn^{1/k})k}{4}\right) -\frac{(n - 16pkn^{1/k})k}{4}  - 64p^2k^2n^{2/k}+ 8pkn^{1+1/k}\\   -&  16(k-1)n^{1+1/k} - 16pkn^{1/k}\left(\frac{1}{2} - \frac{1}{p}\right)n 
\\ \geq &    \left(\dfrac{(n - 16pkn^{1/k})^2}{4} + \frac{(n - 16pkn^{1/k})k}{4}\right)  - \frac{nk}{4} - 64p^2k^2n^{2/k}  + 16n^{1+1/k} \\
>& \left(\dfrac{(n - 16pkn^{1/k})^2}{4} + \frac{(n - 16pkn^{1/k})k}{4}\right) ,
\end{align*}
which contradicts \eqref{turan number of wheel bound}. 
\end{proof}

For any vertex $v$, and any subset $A \subset V$, let $d_A(v) = |N(v) \cap A|$. Also, let $\theta>0$ be arbitrary and define \[M:=\{v \in S: d_S(v) \geq \theta n\} \cup \{v \in T: d_T(v) \geq \theta n\}.\]
We will now see that $M$ and $P$ are empty sets. The following lemmas will prove this.

\begin{lemma}
\label{lemma6}\label{lemma9}
Let $\epsilon>0$ be arbitrary. 
 Then
\[|M| \leq \frac{3\epsilon n}{\theta}\] and $M \setminus P$ is empty.
\end{lemma}

\begin{proof}
We know from Lemma \ref{lemma first maxcut} that (for $n$ large enough) $G$ has a maximum cut with $e(S, T) \geq \bigg(\dfrac{1}{4} - \epsilon\bigg)n^2$.
Hence, for $k = 2$, \[e(S) + e(T) = e(G) - e(S, T) \leq \frac{n^2}{4} + \frac{n}{2} - \frac{n^2}{4} + \epsilon n^2 \leq \frac{n}{2} + \epsilon n^2,\]
and for $k \geq 3$ and $n$ large enough,
\[e(S) + e(T) = e(G) - e(S, T) \leq \frac{n^2}{4} + \frac{n(k-1)}{4} + \frac{(k-1)^2}{16} + 1 - \frac{n^2}{4} + \epsilon n^2 \leq \frac{3}{2}\epsilon n^2.\]

On the other hand, if we let $M_1 = M \cap S$ and $M_2 = M\cap T$, then
\[2 e(S) = \sum_{u\in S}d_S(u) \geq \sum_{M_1}d_S(u) \geq |M_1|\theta n\]

\[2 e(T) = \sum_{u\in T}d_T(u) \geq \sum_{M_2}d_T(u) \geq |M_2|\theta n\]

So, $e(S) + e(T) \geq \frac{|M|\theta n}{2}$, and hence
$\frac{|M|\theta n}{2} \leq  \frac{3 \epsilon n^2}{2}$. Therefore proving,
$|M| \leq  \frac{3 \epsilon n}{\theta}$.

We now prove that $M \setminus P$ is empty. Let us call $P_1 = P \cap S$, and $P_2 = P\cap T$. Suppose $M\setminus P \neq \emptyset$ . We assume without loss of generality that there exists a vertex $u \in M_1 \setminus P_1$. As $S$ and $T$ form a maximum cut, $d_T(u) \geq \frac{d(u)}{2}$. Also, since $u \not\in P$, it follows that $d(u) \geq  \bigg( \frac{1}{2} - \frac{1}{p} \bigg)n
$. 
Therefore, $d_T(u) \geq  \bigg( \frac{1}{4} - \frac{1}{2p} \bigg)n$. On the other hand $|P| \leq 16pkn^{1/k}$. Hence, for fixed $\epsilon, \theta, p$ and for $n$ large enough, we have

\begin{equation}
\label{eqn25}
  |S\setminus (M \cup P)| \geq \bigg(\dfrac{1}{2} - \sqrt{\epsilon}\bigg)n - \dfrac{3 \epsilon n}{\theta} - 16pkn^{1/k} > k.
\end{equation}

Now suppose that $u$ is adjacent to $k$ distinct vertices $u_1, \ldots, u_k \in S\setminus(M \cup P)$. Since $u_i \not\in P$, we have \[d(u_i) \geq \bigg( \dfrac{1}{2} - \dfrac{1}{p} \bigg)n\]
On the other hand $d_S(u_i) \leq \theta n$. So,
\[d_T(u_i) = d(u_i) - d_S(u_i) \geq \bigg( \dfrac{1}{2} - \dfrac{1}{p} \bigg)n - \theta n\]
By Lemma \ref{inclusion exclusion lemma} we have,
\begin{equation}
\label{eqn26}
\begin{split}
|N_T(u) \cap N_T(u_1) \cap \cdots \cap N_T(u_k)| &\geq |N_T(u)| + |N_T(u_1)|+ \cdots + N_T(u_k)| - k|N_T(u) \cup N_T(u_1) \cup \ldots \cup N_T(u_k)|\\
& \geq \left(\frac{1}{4} - \frac{1}{2p}\right)n + k\left(\frac{1}{2} - \frac{1}{p} - \theta\right)n - k|T|\\
& \geq \bigg(\dfrac{1}{4} - \dfrac{2k + 1}{2p} -k\theta - k\sqrt{\epsilon}\bigg)n\\
& > k,
\end{split}
\end{equation}
where the last inequality holds if we choose $p>20k$, $\theta< \frac{1}{20k}$, $\epsilon < \frac{1}{100k^2}$ and $n$ large enough. This implies that there are at least $k$ distinct vertices $v_1, \ldots, v_k \in T$ such that $\{v_1, \ldots, v_k\} \subseteq |N_T(u) \cap N_T(u_1) \cap \ldots \cap N_T(u_k)|$.
This is a contradiction as G should not contain a $W_{2k+1}$. Therefore, $u$ can be adjacent to at most $k-1$ vertices in $S\setminus (M\cup P)$.
Therefore, 
\begin{equation}
\label{eqn27}
\begin{split}
d_S(u) &\leq |M| + |P| + k-1\\
&\leq \dfrac{3\epsilon n}{\theta} + 16pkn^{1/k} + k-1\\
& < \theta n,
\end{split}
\end{equation}
where the last inequality holds by choosing $\epsilon < \frac{\theta^2}{6}$ and $n$ large enough. This contradicts $u \in M$ and therefore $M_1 \setminus P_1$ must be empty, and hence $M \setminus P = \emptyset$.
\end{proof}

\begin{lemma}
\label{lemma5.4}
The set $P$ is empty. $G[S]$ and $G[T]$ are $K_{1, k}$-free.
\end{lemma}

\begin{proof}
In the proof for Lemma \ref{lemma9} we showed that there are no vertices in $M_1 \setminus P$ (or $M_2 \setminus P$) adjacent to at least $k$ vertices in $S\setminus P$ (or $T\setminus P$, respectively). We will similarly show that $G[S\setminus P]$ and $G[T\setminus P]$ are $K_{1,k}$-free.
Without loss of generality assume to the contrary that there exists a vertex $u \in S\setminus P$ that is adjacent to $k$ distinct vertices $u_1, \ldots, u_k$ in $S \setminus P$. Then 
\begin{equation}
\label{eqn28}
\begin{split}
|N_T(u) \cap N_T(u_1) \cap \ldots \cap N_T(u_k)| &\geq |N_T(u)| + |N_T(u_1)|+ \ldots + |N_T(u_k)| - k|N_T(u) \cup N_T(u_1) \cup \ldots \cup N_T(u_k)|\\
&\geq (k+1)\bigg( \dfrac{1}{2} - \dfrac{1}{p} - \theta\bigg)n - k\bigg(\dfrac{1}{2} + \sqrt{\epsilon}\bigg)n\\
& = \bigg(\dfrac{1}{2} - \dfrac{k+1}{p} - (k+1)\theta - k\sqrt{\epsilon}\bigg)n\\
& > k
\end{split}
\end{equation}
for sufficiently large $n$ and $p$ and sufficiently small $\theta$ and $\epsilon$. This implies that there are at least $k$ distinct vertices $v_1, \ldots, v_k \in T$ such that $\{v_1, \ldots, v_k\} \subseteq |N_T(u) \cap N_T(u_1) \cap \ldots \cap N_T(u_k)|$.
This is a contradiction as $G$ should not contain a $W_{2k + 1}$. Therefore, $u$ can be adjacent to at most $k-1$ vertices in $S\setminus P$. This implies that $G[S\setminus P]$ is $K_{1,k}$-free and similarly $G[T\setminus P]$ is $K_{1,k}$-free. 

Next, let $z$ be a vertex of G with largest eigenvector entry. By possible rescaling we may assume that $x_z = 1$.
Therefore, \begin{equation}
    \label{eqn29}
    \begin{split}
        d(z) \geq \sum_{v \sim z}x_v = \lambda_1 x_z = \lambda_1 > \frac{n+k-1}{2},
    \end{split}
\end{equation}
and hence $z \not\in P$. Assume without loss of generality then, that $z \in S$.
Also, since $G(S\setminus P)$ is $K_{1, k}$ free, 
\[d_S(z) = d_{S\setminus P}(z) + d_{S\cap P}(z) \leq k-1 + |S\cap P|\]

So,
\begin{equation}
    \label{eqn30}
    \begin{split}
 \lambda_1 &= \lambda_1 \mathbf{x}_z = \sum_{v \sim z}\mathbf{x}_v \\
 &= \sum_{\substack{v \sim z \\ v \in S}} \mathbf{x}_v +  \sum_{\substack{v \sim z \\ v \in T}}\mathbf{x}_v \\
 &= \sum_{\substack{v \sim z \\ v \in S}}\mathbf{x}_v + \sum_{\substack{v \sim z \\ v \in P_2}}\mathbf{x}_v + \sum_{\substack{v \sim z \\ v \in T\setminus P_2}}\mathbf{x}_v \\
 & \leq d_S(z) + |P_2| + \sum_{v \in T\setminus P} \mathbf{x}_v\\
 & \leq k-1 + |S\cap P| + |T\cap P| + \sum_{v \in T\setminus P} \mathbf{x}_v\\
 & \leq k-1 + 16pkn^{1/k} + \sum_{v \in T\setminus P} \mathbf{x}_v.\\
    \end{split}
\end{equation}
Therefore,
\begin{equation}
    \label{eqn31}
    \sum_{v \in T\setminus P} x_v \geq \lambda_1 - 16pkn^{1/k} - k+1
\end{equation}

Now to show $P = \emptyset$, first assume to the contrary that there exists some vertex $v \in P$ with $d(v) \leq \bigg( \dfrac{1}{2} - \dfrac{1}{p} \bigg)n$. Then consider the modified graph, $G^+$ with vertex set $V(G)$ and edge set $E(G^+) = E(G \setminus \{v\}) \cup \{vw : w \in T\setminus P\}$. That is, effectively we are deleting the vertex $v$ and replacing it with another vertex that is adjacent to all the vertices in the set $T\setminus P$. This modification of $G$ to $G^+$ preserves the property of being $W_{2k+1}$-free. If a wheel, $W_{2k+1}$, would be created after the modification, then either (i) $v$ is the centre of the wheel, or (ii) $v$ is in the cycle part of the wheel. In the first case if $v$ were the centre of a wheel, then it would have $2k$ neighbours in $T\setminus P$ that induce a cycle. We can show, by choosing $p$ sufficiently large, that the $2k$ vertices would already have had a common neighbour in $S$ in this case, and therefore such a case would not be possible to begin with. 
On the other hand, in the second case, if $v \in N(c)$ where $c$ denotes the centre of the wheel created. 
Then $v$ must be adjacent to at least two other vertices $c_1$ and $c_2$ in $T \setminus P$. Again choosing $p$ large enough shows that $|N_S(c_1) \cap N_S(c_2) \cap N_S(c)| \geq 2k - 2$, and therefore, if the modification contained a $W_{2k+1}$ then $G$ itself would have contained a $W_{2k+1}$, which is a contradiction. Thus, $G^+$ is $W_{2k+1}$-free.

Now, using equation (\ref{eqn31}) we can say that
\begin{equation}
    \begin{split}
        \lambda_1(G^+) - \lambda_1(G) &\geq \dfrac{\mathbf{x}^T (A(G^+) - A(G)) \mathbf{x}}{\mathbf{x}^T\mathbf{x}} = \dfrac{2 \mathbf{x}_v}{\mathbf{x}^T\mathbf{x}}\bigg(\sum_{w \in T\setminus P} x_w - \sum_{vw \in E(G)} x_w\bigg)\\
        & \geq \dfrac{2 \mathbf{x}_v}{\mathbf{x}^T\mathbf{x}} \bigg(\lambda_1 -16pkn^{1/k} - k+1 - d_G(v)\bigg)\\
        & > \dfrac{2 \mathbf{x}_v}{\mathbf{x}^T\mathbf{x}} \bigg(\dfrac{n+k-1}{2} - 16pkn^{1/k} - k+1 - \bigg( \dfrac{1}{2} - \dfrac{1}{p} \bigg)n\bigg)\\
        & = \dfrac{2 \mathbf{x}_v}{\mathbf{x}^T\mathbf{x}}\bigg(\dfrac{n}{p} -16pkn^{1/k} - \frac{k}{2} + \frac{1}{2}\bigg)\\
        & > 0
    \end{split}
\end{equation}
for $n$ large enough. This contradicts the fact that $G$ has maximum spectral radius over all $W_{2k+1}$-free graphs. Hence, the set of atypical vertices, $P$, must be empty. Moreover, it follows from here that $G[S\setminus P]= G[S]$ and $G[T\setminus P] = G[P]$ are $K_{1,k}$-free.

\end{proof}

\begin{lemma}
\label{lemma5.5}
We have the following bounds on the sizes of $|S|$ and $|T|$.

\[\frac{n}{2} - \sqrt{\frac{3nk}{2}} \leq |S|, |T| \leq \frac{n}{2}+\sqrt{\frac{3nk}{2}}.\]
\end{lemma}

\begin{proof}
We know from Lemma \ref{cor1} that $e(G) \geq \lambda_1^2 - \frac{6t}{n}$. Using Lemma \ref{lemma5.4} we may obtain an improved upper bound on the number of triangles in G, to obtain a lower bound on $e(G)$.
\begin{equation}
    \begin{split}
       t & \leq \frac{n_0(k-1)}{2}\bigg(\frac{k-2}{3}+ (n-n_0)\bigg) + \frac{(n-n_0)(k-1)}{2}\bigg(\frac{k-2}{3}+ n_0\bigg)\\
       & = \frac{n(k-1)(k-2)}{6} +  \bigg(\frac{n^2}{4} - q^2\bigg)(k-1)
    \end{split}
\end{equation}
where $n_0 = \frac{n}{2} + q$ is the size of the larger part. The above upper bound may be obtained by observing that any triangle in $G$ contains an edge in either $S$ or $T$ and two more edges, which either lie in the same part or in $E(S,T)$. Note that by Lemma \ref{lemma5.4} the vertices of any edge is $S$ or $T$ have at most $k-2$ common neighbours and any triangle lying entirely in one of the parts would be counted thrice depending on which of the three edges we chose to begin with initially. 

We also have from \eqref{lower bound G} that $\lambda_1 >  \frac{n+k-1}{2}$ for $n$ large enough.

Therefore, 
\begin{equation}
    \begin{split}
      e(G) & > \bigg(\frac{n}{2} + \frac{k-1}{2}\bigg)^2 - \frac{6}{n}\bigg(\frac{n(k-1)(k-2)}{6} +  \bigg(\frac{n^2}{4} - q^2\bigg)(k-1)\bigg)\\
       & = \frac{n^2}{4} + \frac{n(k-1)}{2} +  \frac{k^2 - 2k + 1}{4} - \frac{3n(k-1)}{2} - (k-1)(k-2) + \frac{6q^2(k-1)}{n}\\
       & \geq \frac{n^2}{4} - n(k - 1) - \frac{3k^2 - 10k + 7}{4} 
    \end{split}
\end{equation}
On the other hand, $e(G) = e(S) + e(T) + e(S, T) \leq \frac{n(k-1)}{2} + \frac{n^2}{4} - q^2$. So, 
\[\frac{n(k-1)}{2} + \frac{n^2}{4} - q^2  \geq \frac{n^2}{4} - n(k - 1) - \frac{3k^2 - 10k + 7}{4} \]
and therefore,
\[\frac{3n(k-1)}{2} + \frac{3k^2 - 10k + 7}{4} \geq q^2,\]
\[\frac{3nk}{2} > q^2\] for $n$ large enough.
Thus giving $q < \sqrt{\frac{3nk}{2}}$. 

\end{proof}

We now show that all of the eigenvector entries are close to the maximum. 
\begin{lemma}
\label{lem5.6}
For all $u \in V(G)$ and $0 < \epsilon' < 1$, we have $x_u > 1 - \epsilon'$.
\end{lemma}
\begin{proof}
Recall that the vertex $z$ is the vertex with largest eigenvector entry, $x_z = 1$. Without loss of generality, assume that $z \in S$. Since $d(z) \geq \lambda_1 > \frac{n+k-1}{2}$, and $d(z) = d_T(z) + d_S(z)$, then $d_T(z) > \frac{n+k-1}{2} - (k-1) = \frac{n-k+1}{2}$. This implies that $|T| >\frac{n-k+1}{2}$ and $|S| < \frac{n+k-1}{2}$. Since $d_S(z) \leq k-1$, the amount of eigenweight in neighbourhood of $z$ lying in $T$ is greater than or equal to $\lambda_1 - (k-1) > \frac{n-k+1}{2}$. 

Let $z'$ be an arbitrary vertex of $S$. Then by lemma \ref{lemma5.4}, $d_T(z') \geq \frac{n}{2} - \frac{n}{p} - (k-1)$.
We lower bound the amount of eigenweight lying in the common neighbourhood of $z$ and $z'$ by noting that an upper bound for the eigenweight of vertices in $N_T(z)$ but not in $N_T(z')$ is given by $|T| - d_T(z')$.
Therefore,
\begin{equation}
    \begin{split}
    \text{Eigenweight in } N_T(z) \cap N_T(z')  &= \sum_{u\in N_T(z) \cap N_T(z')} \mathbf{x}_u \\ & \geq \left(\sum_{u\in N_T(z)} \mathbf{x}_u \right) - (|T| - d_T(z'))\\
    &> \left(\frac{n-k+1}{2}\right) - \left(|T| - d_T(z')\right)\\
    &\geq \left(\frac{n-k+1}{2}\right) - \left( \left(\frac{n}{2} + \sqrt{\frac{3nk}{2}}\right) - \left(\frac{n}{2} - \frac{n}{p} - (k-1)\right)\right)\\
    & = \frac{n-3k+3}{2}  - \frac{n}{p} - \sqrt{\frac{3nk}{2}} \\
    & \geq \frac{n}{2} - \frac{n}{p} - \sqrt{2nk},
    \end{split}
\end{equation}
for $n$ large enough. Therefore, 
\begin{equation}
 \lambda_1 \mathbf{x}_{z'}  = \sum_{u\sim z'}\mathbf{x}_u
                   \geq\frac{n}{2} - \frac{n}{p} - \sqrt{2nk}.
\end{equation}
Since $G$ is a subgraph of the union of a complete bipartite graph and a graph of maximum degree $k-1$, we have the upper bound, $\lambda_1 \leq \frac{n}{2} + k-1$. Therefore,
\begin{equation}
    \begin{split}
 \mathbf{x}_{z'} & \geq \frac{\frac{n}{2} - \frac{n}{p} - \sqrt{2nk}}{\lambda_1}\\
        & \geq  \frac{\frac{n}{2} - \frac{n}{p} - \sqrt{2nk}}{n\bigg(\frac{1}{2} + \frac{k-1}{n}\bigg)}.
    \end{split}
\end{equation}
For $n$ and $p$ large enough, this gives us $\mathbf{x}_{z'} > 1 - \frac{\epsilon'}{2}$.

Similarly, let $w$ be an arbitrary vertex in $T$. Then $d_S(w) \geq \frac{n}{2}  - \frac{n}{p} - (k-1)$. Now, since every vertex in $S$ has eigenvector entry greater than $1 - \frac{\epsilon'}{2}$, it implies that \begin{equation}
    \begin{split}
 \mathbf{x}_{w} & \geq \frac{\left(\frac{n}{2} - \frac{n}{p} - (k-1)\right)(1 - \frac{\epsilon'}{2})}{\frac{n}{2} + k-1}\\
        & \geq  \frac{n\left(\frac{1}{2} - \frac{1}{p} - \frac{k-1}{n}\right)(1 - \frac{\epsilon'}{2}) }{n\bigg(\frac{1}{2} + \frac{k-1}{n}\bigg)}\\
       > 1 - \epsilon'
    \end{split}
\end{equation}
for $n$ and $p$ large enough.
\end{proof}

\section{The proof of Theorem \ref{thm2}}\label{section k2}

To prove Theorem \ref{thm2}, assume that $G$ is a graph in $\mathrm{SPEX}(n, W_5)$, and that $n$ is large enough. We will show that $e(G) = \textup{ex}(n, W_5)$. Let $S$ and $T$ be the two parts of a maximum cut of $G$ as in Section \ref{structural lemmas section}. By Lemma \ref{lemma5.4} we know that the two induced graphs, $G[S]$ and $G[T]$ are matchings. If one could increase the size of the matching in $G[S]$ or $G[T]$, this would not create any $W_5$ and would strictly increase the spectral radius. Therefore, $G$ must have that $G[S]$ and $G[T]$ are matchings of size $\lfloor \frac{|S|}{2}\rfloor$ and $\lfloor\frac{|T|}{2}\rfloor$ respectively. Similarly, we must have that 
\[E(S,T) = \bigg\{\{u, v\}  \text{ for all }u \in S, v \in T\bigg\}.\]
This is again because adding more edges to $E(S,T)$, if possible, will not create a $W_5$, but strictly increases the spectral radius.
Say $|S| \leq |T|$ and let $|S| = \frac{n}{2}-q$ and $|T| = \frac{n}{2} + q$. We will now argue that $q\leq 1$ and so $|T| \leq |S| + 2$. For this we will use lower bounds on $\lambda_1(G)$, obtained similarly to how they were found in equation (\ref{eqn1}), and upper bounds on $\lambda_1(G)$. Let $H$ be any graph in $\mathrm{EX}(n, W_5)$. 

We break this argument into two cases based on when $n \equiv 0 \pmod 4$ and when $n \not\equiv 0 \pmod 4$.

\begin{case}[\textbf{When $n \equiv 0 \pmod 4$}]
\begin{equation}
\label{eqn15a}
    \lambda_1(G) \geq \lambda_1(H) \geq \dfrac{\mathbf{1}^T A(H)\mathbf{1}}{\mathbf{1^T 1}} = 2\dfrac{\bigg\lfloor\dfrac{n^2}{4}\bigg\rfloor + \bigg\lfloor\dfrac{n}{2} \bigg\rfloor}{n} = 2\dfrac{\bigg(\dfrac{n^2}{4}+\dfrac{n}{2}\bigg)}{n} = \dfrac{n}{2} + 1.
\end{equation}
On the other hand, 
\begin{equation}
\label{eqn15}
    \lambda_1(G) \leq \sqrt{|S||T|} + 1 \leq \sqrt{\bigg(\dfrac{n}{2} - q\bigg)\bigg(\dfrac{n}{2} + q\bigg)} + 1 = \sqrt{\bigg(\dfrac{n^2}{4} - q^2\bigg)} + 1
\end{equation} where $q \in \mathbb{N}\cup \{0\}$.
So we have,
$\dfrac{n}{2} + 1 \leq \lambda_1(G) \leq \sqrt{\bigg(\dfrac{n^2}{4} - q^2\bigg)} + 1$. Therefore $q = 0$, meaning $|S| = |T|$.
\end{case}
\begin{case}[\textbf{When $n \not\equiv 0 \pmod 4$}]
\begin{equation}
\label{eqn15b}
   \lambda_1(G) \geq \lambda_1(H) \geq \dfrac{\mathbf{1}^T A(H)\mathbf{1}}{\mathbf{1^T 1}} \geq 2\dfrac{\bigg(\dfrac{n^2}{4}+\dfrac{n}{2} - 1\bigg)}{n} = \dfrac{n}{2} + 1 - \dfrac{2}{n}.
\end{equation}
On the other hand, 
\begin{equation}
\label{eqn15''}
    \lambda_1(G) \leq \sqrt{|S||T|} + 1 \leq \sqrt{\bigg(\dfrac{n}{2} - q\bigg)\bigg(\dfrac{n}{2} + q\bigg)} + 1 = \sqrt{\bigg(\dfrac{n^2}{4} - q^2\bigg)} + 1
\end{equation} where $2q \in \mathbb{N}\cup \{0\}$.
So we have,
$\dfrac{n}{2} + 1 - \dfrac{2}{n} \leq \lambda_1(G) \leq \sqrt{\bigg(\dfrac{n^2}{4} - q^2\bigg)} + 1$. Which gives $\dfrac{n^2}{4} - 2 + \dfrac{4}{n^2} \leq \dfrac{n^2}{4} - q^2$, meaning $q^2 \leq 2 - \dfrac{4}{n^2}$, and so $q < \sqrt{2}$.
This implies that when $n$ is odd, then $q = 0.5$ and we obtain that $|T| = |S|  + 1$; and when $n \equiv 2 \pmod 4$, then $q \leq 1$ and $|T| \leq |S|  + 2$. 
\end{case}
In fact, for $n \equiv 2 \pmod 4$,  we may explicitly calculate the largest eigenvalues of the cases when $|T| = |S|$ (balanced case), and $|T| = |S| + 2$ (unbalanced case) and observe that the spectral radius is maximized in the unbalanced case. One may observe this by calculating the spectral radius of their respective equitable matrices, \[B = \begin{bmatrix}
\frac{n}{2} & 1 \\
\frac{n}{2} - 1 & 1
\end{bmatrix} \text{and } U = \begin{bmatrix}
1 & \frac{n}{2} - 1 \\
\frac{n}{2} + 1 & 1
\end{bmatrix},\] where the two parts of the balanced case ($B$) are the set of vertices that have internal degree $1$, and the set of vertices that have internal degree $0$;  and the two parts of the unbalanced case ($U$) are the sets $S$ and $T$.

In all cases, for large enough $n$, $\mathrm{SPEX}(n, W_5) \subseteq \mathrm{EX}(n, W_5)$ , thus proving Theorem \ref{thm2}.

\section{The proof of Theorem \ref{thm4}}\label{section k3}

In this section, we assume that $k\geq 3$ and that $G \in \mathrm{SPEX}(n, W_{2k+1})$. We recall the notation from previous sections. $S$ and $T$  denote the two parts in the maximum cut of $G$. Define the {\em internal degree} of a vertex $u$ to be $d_S(u)$ if $u\in S$ or $d_T(u)$ if $u\in T$. By Lemma \ref{lemma5.4} every vertex has internal degree at most $k-1$ and degree at least $\left(\frac{1}{2}-\frac{1}{p}\right)n$ where we may choose $p$ to be a constant large enough for our needs. We have the following lemma related to the set of edges, $E(S, T)$.
\begin{lemma}
\label{lem6}
Every vertex, $u$ in $S$ (or $T$), with `internal degree' $0 \leq d \leq k-1$, is adjacent to all but at most $d$ vertices in $T$ (or $S$ respectively.) 
\end{lemma}
\begin{proof}

The proof is same for $u \in S$ as $u \in T$, so we will prove it in the case $u$ is an arbitrary vertex in $S$ only. Let $u$ have internal degree $d$ and by Lemma \ref{lemma5.4} we have $0 \leq d \leq k-1$. Let $\{u_1, u_2, \ldots, u_d\}$ be its neighbors in $G[S]$. If we now modify the graph $G$ to $G'$ by deleting the edges $\{u, u_i\}$ for all $1 \leq i \leq d$, and adding edges until $u$ is adjacent to all vertices in $T$, then we claim that $G'$ still does not have any $W_{2k+1}$.

$G'$ has no $W_{2k+1}$ because if $u$ were the `center' of a new cycle, then that would imply that $G[T]$ has $C_{2k}$ as a subgraph. This is not possible because each vertex of the cycle has degree at least $(1/2 - 1/p)n$ and so the $2k$ vertices of the cycle would have at least one vertex in common in $S$, implying that $G$ already contained $W_{2k+1}$. On the other hand, if $u$ were part of the $C_{2k}$ of a $W_{2k+1}$ in $G'$, then $u$ would be adjacent to the centre $c$ and two more vertices $v_1$ and $v_2$ of the $W_{2k+1}$, all lying in $T$, in $G'$. If this were the case, $c, v_1,$ and $v_2$ would similarly already be adjacent to at least $2k - 2$ vertices in $S$, in $G$. This is not possible again because then $G$ would have already had a $W_{2k+1}$. Hence, $G'$ is $W_{2k+1}$-free.

Now since, $\lambda_1(G) \geq \lambda_1(G')$, it implies that 
\begin{equation}
\label{eqn39}
    \begin{split}
       0 \leq \lambda_1(G) - \lambda_1(G') &\leq \dfrac{\mathbf{x}^T (A(G) - A(G')) \mathbf{x}}{\mathbf{x}^T\mathbf{x}} = \dfrac{2 \mathbf{x}_u}{\mathbf{x}^T\mathbf{x}}\bigg(\sum_{i = 1}^d x_{u_i} - \sum_{\substack{uw \not\in E(G)\\ w \in T }} x_w\bigg)\\
        & \leq \dfrac{2 \mathbf{x}_u}{\mathbf{x}^T\mathbf{x}} \bigg(d - (1 - \epsilon)|W|\bigg)
    \end{split}
\end{equation}
where $W = \{w : uw \not\in E(G), w \in T\}$.

Therefore, $d - (1 - \epsilon)(|W|) \geq 0$. Choosing $\epsilon < \frac{1}{k}$ implies $|W| \leq d$. 

\end{proof}

We can use Lemma \ref{lem6}, to say the following. Without loss of generality choose one of the parts, say $S$, and a set  of vertices,  $F \subset S$. If $|F| = f$, and every vertex in $F$ has internal degree at most $d$, then for any $R\subset T$ there is a subset of $R$ of size $|R| - fd$, such that each vertex in this set is adjacent to all the vertices of $F$.  

Our next goal is to show that there is a vertex of internal degree $k-1$ (Lemma \ref{lem9}). In order to prove this, we use the following lemma that allows us to control the density inside each part.

\begin{lemma}\label{general interlacing use}
Assume that $C$ is a constant and that at most $C$ edges may be removed from $G$ so that the graph induced by $S$ has maximum degree $a$ and the graph induced by $T$ has maximum degree $b$. Then for $n$ large enough we must have $a+b \geq k-1$. If $C = 0$ we must have $a+b \geq k$.
\end{lemma}

\begin{proof}
By \eqref{lower bound G} and since $G$ is extremal, for $n$ large enough we must have $\lambda_1(G) > \frac{k-1 + \sqrt{(k-1)^2+n^2-1}}{2} + \frac{1}{2n}$. On the other hand, if $G$ is the subgraph of a graph of maximum degree $a$ joined to a graph of maximum degree $b$ plus at most $C$ edges, then we have by Lemmas \ref{thm5} and \ref{lem5.6}
\[
\lambda_1(G) = \frac{\mathbf{x}^TA(G)\mathbf{x}}{\mathbf{x}^T\mathbf{x}} \leq \lambda_1\left( \begin{bmatrix} a & |T| \\ |S| & b \end{bmatrix}\right) + \frac{2C}{\mathbf{x}^T\mathbf{x}} \leq \frac{a+b +\sqrt{(a+b)^2 + n^2}}{2} + \frac{2C}{n(1-\epsilon)^2}.
\]
For $n$ large enough, combining the two inequalities gives that for $C=0$ we must have $a+b\geq k$ and otherwise $a+b \geq k-1$.
\end{proof}

\begin{lemma}
\label{lem9}
For $k \geq 3, k\not\in \{4,5\},$ there exists at least one vertex in $G[S]$ or $G[T]$ with degree equal to $k-1$.
\end{lemma}

\begin{proof}
To prove this lemma, it suffices to show that there exists a vertex in $G[S]$ or $G[T]$, with degree at least $k-1$.
We prove this lemma by recursively applying Lemma \ref{general interlacing use} to show the existence of a vertex with higher and higher degrees in either $G[S]$ or $G[T]$. We begin by proving the following claim. 
\begin{claim}
\label{lem7}
There exists at least one vertex in $G[S]$ or $G[T]$ with degree at least $\frac{k}{2}$. 
\end{claim}
\begin{proof}

Assume to the contrary that there do not exist any such vertices in $G$. Then every vertex has `internal degree' at most $\frac{k-1}{2}$. Then $G$ must be a subgraph of some graph $H$ of the form of Lemma \ref{general interlacing use}, where $H$ is the join of a graph $H_1$ with maximum degree $a$ and another graph $H_2$ with maximum degree $b$, where $a = b \leq \frac{k-1}{2}$, and $n_0 = |S|$.
Then $a + b \leq k-1$, which contradicts Lemma \ref{general interlacing use}, since $C = 0$. Hence, there must exist a vertex in $G[S]$ or $G[T]$ with degree greater than $\frac{k-1}{2}$. 
\end{proof}

It follows from Claim \ref{lem7} that if $k$ is odd, then in fact there must exist a vertex with `internal degree' at least $\frac{k+1}{2}$. This proves the lemma for $k=3$.
Now to be precise about which part contains a vertex of large internal degree, we will use the following notation.
Let $L \in \{S, T\}$ be a part of $G$ such that $G[L]$ has a vertex, $v$, of degree at least $\frac{k}{2}$. Let $R:= L^{\mathsf{c}}$. 

Let $ \mathcal{N} = \{v_1, v_2, \ldots, v_{\ceil{\frac{k}{2}}}\}$, be a set of $\ceil{\frac{k}{2}}$ distinct vertices in $N_L(v)$. Then all the the vertices in $\mathcal{N} \cup \{v\}$ must be adjacent to a set $R' \subset R$, of minimum size $|R| - \ceil{\frac{k+2}{2}}(k-1)$.  

\tikzset{every picture/.style={line width=0.75pt}} 
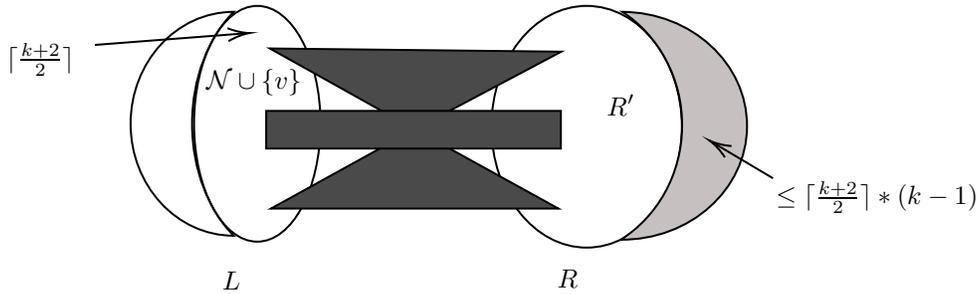
\begin{figure}[h!]
    \centering
    
\begin{tikzpicture}[x=0.75pt,y=0.75pt,yscale=-1,xscale=1]

\draw   (106.23,115.2) .. controls (106.23,82.34) and (120.51,55.7) .. (138.12,55.7) .. controls (155.73,55.7) and (170,82.34) .. (170,115.2) .. controls (170,148.06) and (155.73,174.7) .. (138.12,174.7) .. controls (120.51,174.7) and (106.23,148.06) .. (106.23,115.2) -- cycle ;
\draw   (256.23,116.7) .. controls (256.23,83.01) and (277.73,55.7) .. (304.24,55.7) .. controls (330.76,55.7) and (352.26,83.01) .. (352.26,116.7) .. controls (352.26,150.39) and (330.76,177.7) .. (304.24,177.7) .. controls (277.73,177.7) and (256.23,150.39) .. (256.23,116.7) -- cycle ;
\draw  [fill={rgb, 255:red, 198; green, 192; blue, 192 }  ,fill opacity=1 ] (322.23,58.7) .. controls (357.18,58.93) and (385.34,84.86) .. (385.13,116.62) .. controls (384.92,148.38) and (356.42,173.93) .. (321.47,173.7) .. controls (339.86,162.04) and (352.22,140.78) .. (352.38,116.4) .. controls (352.55,92.03) and (340.47,70.6) .. (322.23,58.7) -- cycle ;
\draw    (398.23,142.7) -- (367.95,124.72) ;
\draw [shift={(366.23,123.7)}, rotate = 390.7] [color={rgb, 255:red, 0; green, 0; blue, 0 }  ][line width=0.75]    (10.93,-3.29) .. controls (6.95,-1.4) and (3.31,-0.3) .. (0,0) .. controls (3.31,0.3) and (6.95,1.4) .. (10.93,3.29)   ;
\draw  [fill={rgb, 255:red, 74; green, 74; blue, 74 }  ,fill opacity=1 ] (217.52,118) -- (290.02,158) -- (145.02,158) -- cycle ;
\draw  [fill={rgb, 255:red, 74; green, 74; blue, 74 }  ,fill opacity=1 ] (217.52,118) -- (144.45,77.19) -- (291.46,78.81) -- cycle ;
\draw  [color={rgb, 255:red, 0; green, 0; blue, 0 }  ,draw opacity=1 ][fill={rgb, 255:red, 74; green, 74; blue, 74 }  ,fill opacity=1 ] (142.62,108.7) -- (291.23,108.7) -- (291.23,127.7) -- (142.62,127.7) -- cycle ;
\draw   (126.23,171.7) .. controls (97.51,171.7) and (74.23,146.4) .. (74.23,115.2) .. controls (74.23,84) and (97.51,58.7) .. (126.23,58.7) .. controls (113.49,71.53) and (105.23,92.06) .. (105.23,115.2) .. controls (105.23,138.34) and (113.49,158.87) .. (126.23,171.7) -- cycle ;
\draw    (56.23,75.4) -- (132.24,69.55) ;
\draw [shift={(134.23,69.4)}, rotate = 535.6] [color={rgb, 255:red, 0; green, 0; blue, 0 }  ][line width=0.75]    (10.93,-3.29) .. controls (6.95,-1.4) and (3.31,-0.3) .. (0,0) .. controls (3.31,0.3) and (6.95,1.4) .. (10.93,3.29)   ;

\draw (398.23,142.7) node [anchor=north west][inner sep=0.75pt]    {$ \leq \lceil \frac{k+2}{2} \rceil *( k-1)$};
\draw (119,188) node [anchor=north west][inner sep=0.75pt]   [align=left] {$\displaystyle L$};
\draw (288,187) node [anchor=north west][inner sep=0.75pt]   [align=left] {$\displaystyle R$};
\draw (10,73) node [anchor=north west][inner sep=0.75pt]    {$\lceil \frac{k+2}{2} \rceil $};
\draw (313,99) node [anchor=north west][inner sep=0.75pt]   [align=left] {$\displaystyle R'$};
\draw (111,86) node [anchor=north west][inner sep=0.75pt]   [align=left] {$\displaystyle \mathcal{N}\cup \{v\}$};

\end{tikzpicture}
\caption{$\mathcal{N}\cup \{v\}$ is a subset of $L$ which is adjacent to $R'\subset R$ of size at least $|R| - \lceil\frac{k+2}{2}\rceil(k-1)$}
    \label{fig:my_label}
\end{figure}

Next, observe that $G[R']$ has at most $\floor{\frac{3k-2}{2}}$ vertices in any $\ceil{\frac{k}{2}}$ disjoint paths. This is because $C_{2k} \not\subset G[N(v)] \subset G$. In particular, there cannot be $\ceil{\frac{k}{2}}$ vertex disjoint paths on $3$ vertices or more in $G[R']$. Now, observe that any two vertices $u_1$ and $u_2$ having $d_{R'}(u_i) \geq 2$ and lying at distance $3$ or more from each other in $G[R']$ are contained in two disjoint $P_3$'s. Since there are at most $k^2$ vertices in $G[R']$ at a distance less than or equal to $2$ from any fixed vertex, it implies that there must be less than $\frac{k^3}{2}$ vertices of degree $2$ or more in $G[R']$.   

Thus, $e(G[R']) < \frac{|R'|}{2} + \frac{k^4}{2}$  (where, up to $\frac{|R'|}{2}$ edges may come from a maximal matching in $R'$ and the remaining edges from those unaccounted edges adjacent to the set of vertices with degree at least $2$ in $G[R']$).

Therefore, there are at most $\frac{|R|}{2} + \frac{k^4}{2} + \ceil{\frac{k+2}{2}}(k-1)^2$ many edges in $G[R]$, and at least $|R| - \frac{k^3}{2} - \ceil{\frac{k+2}{2}}(k-1)^2$ vertices in $G[R]$ have degree at most $1$.

Next we prove the following claim while applying the same argument as in Claim \ref{lem7}.
\begin{claim}
\label{lem8}
There exists at least one vertex $v \in L$ with $d_L(v) \geq k-2$
\end{claim}
\begin{proof}
Assume to the contrary that there do not exist any such vertices in $L$.  Then every vertex in $G[L]$ has degree at most $k-3$. 
Obseve that $G$ must be a subgraph of some graph $H$ of the form of Lemma \ref{general interlacing use}, where $H$ is the join of a graph $H_1$ with maximum degree $a$ and another graph $H_2$ with maximum degree $b$; plus at most $\frac{k^4}{2} + \ceil{\frac{k+2}{2}}(k-1)^2$ more edges embedded in $H_2$, where $a=k-3$, $b = 1$, and $n_0 = |L|$. Then $a+b \leq k-2$ with $C = \frac{k^4}{2} + \ceil{\frac{k+2}{2}}(k-1)^2$, which contradicts Lemma \ref{general interlacing use} for $n$ large enough. Thus, there must exist a vertex either in $G[S]$ or $G[T]$ with degree at least $k-2$.
\end{proof}

Now let $\mathcal{N}' = \{v_1, v_2, \ldots, v_{k-2}\}$, be a set of $k-2$ distinct vertices in $N_L(v)$. Then all the the vertices in $\mathcal{N}' \cup \{v\}$ must be adjacent to a set $R'' \subset R$, of minimum size $|R| - (k-1)^2$. It follows from our arguments above that $G[R'']$ has at most $k+1$ vertices in any $k-2$ disjoint paths. Thus, for $k \geq 6$,  $G[R'']$ cannot have $4$ vertex disjoint edges. Lemma \ref{lemma5.4} implies that any edge in $G[R'']$ is adjacent to at most $2(k-2)$ other edges. 

It follows from this that for $k\geq 6$, there must be at most $3(2k-3) + (k-1)(k-1)^2$ many edges in $G[R]$ and at least $|R| - 3(2k-2)  - (k-1)(k-1)^2$ vertices  in $G[R]$ have degree equal to $0$.

Finally, to show the existence of a vertex in $L$ with `internal degree' $k-1$,
assume to the contrary that there do not exist any such vertices in $L$. Then every vertex in $G[L]$ has degree at most $k-2$. Then $G$ must be a subgraph of some graph $H$ of the form of Lemma \ref{general interlacing use}, where $H$ is the join of a graph $H_1$ with maximum degree $a$ and another graph $H_2$ with maximum degree $b$; plus at most $3(2k-3) + (k-1)(k-1)^2$ more edges embedded in $H_2$, where $a=k-2$, $b = 0$, and $n_0 = |L|$. Then $a+b \leq k-2$ with $C = 3(2k-3) + (k-1)(k-1)^2$, which contradicts Lemma \ref{general interlacing use} for $n$ large enough. Thus, there must exist a vertex either in $G[S]$ or $G[T]$ with degree at least $k-1$.
\end{proof}

\begin{lemma}
\label{lem10}
There exist at least $4k^2 + 1$ vertices, $v \in L$, such that $d_L(v) = k-1$ and at most one edge in $G[R]$.
\end{lemma}
\begin{proof}
If there are at most $4k^2$ vertices, $v \in L$, with $d_L(v) = k-1$, then deleting at most $4k^2$ edges of $G[L]$ makes its maximum degree go down to $k-2$. Then, applying Lemma \ref{general interlacing use} gives us a contradiction as $a = k-2$, $b = 0$ and $C \leq 3(2k-3) + (k-1)(k-1)^2 + 4k^2$.  Hence, $L$ has at least $4k^2 + 1$ vertices, $v$, with $d_L(v) = k-1$.

Now suppose $G[R]$ had 2 or more edges, $\{u_1, u_2\}$ and $\{v_1, v_2\}$. Let $\mathcal{I}:= N_L(u_1) \cap N_L(u_2) \cap N_L(v_1) \cap N_L(v_2)$. Then, by Lemmas \ref{lem6} and \ref{lem9}, there are at most $4(k-1)$ vertices in $L \setminus \mathcal{I}$. Now, since there are at least $4k^2+1$ vertices, $c \in L$, with $d_L(c) = k-1$, there exists at least one vertex, $c \in I$, with $d_{\mathcal{I}}(v) = k-1$. This is because there are at most $4(k-1)(k-1) < 4k^2$ vertices $w \in L$, such that $N_L(w) \cup \{w\} \not\subset \mathcal{I}$. Say $\{c_1, c_2, \ldots, c_{k-1}\} = N_L(c)$. Observe that $G[N_R(c) \cap N_R(c_1) \cap \ldots \cap N_R(c_{k-1})] \supset \{u_1, u_2\} \cup \{v_1, v_2\} \sqcup (k-2) K_1$, since $|N_R(c) \cap N_R(c_1) \cap \ldots \cap N_R(c_{k-1})| > |R| - k^2$. Hence, there exist $W_{2k+1} \subset G[\{c\}  \cup N_G(c)] \subset G$, with $c$ as the centre. This is a contradiction. Hence, $G[R]$ has at most one edge.   
\end{proof}
 \begin{lemma}
 \label{lem12}
Let $\mathcal{C} \subset L$, such that $|\mathcal{C}| = c \geq 2$. If $\Hat{G}$ is any graph obtained by modifying $G$ by only changing the edges contained in $E(G[\mathcal{C}])$, such that $e(\Hat{G}[\mathcal{C}]) - e(G[\mathcal{C}]) = m > 0$. Then $\lambda_1(\Hat{G}) - \lambda_1(G) > 0$ for $\epsilon < \frac{m}{c(k-1)}$.
\end{lemma}
\begin{proof}
Take $n$ and $p$ large enough so that $\epsilon < \frac{m}{c(k-1)}$. Then
\begin{equation}
\begin{split}
  \lambda_1(\Hat{G}) - \lambda_1(G) &\geq \frac{\mathbf{x}^T (A(\Hat{G}) - A(G)) \mathbf{x}}{\mathbf{x}^T\mathbf{x}} \geq \frac{2}{\mathbf{x}^T\mathbf{x}}\bigg(e(\Hat{G}[\mathcal{C}])(1 - \epsilon)^2 - e(G[\mathcal{C}])\bigg)\\
        & > \frac{2}{\mathbf{x}^T\mathbf{x}} \bigg(e(\Hat{G}[\mathcal{C}])(1-2\epsilon) - e(G[C])\bigg) = \frac{2}{\mathbf{x}^T\mathbf{x}} \bigg(m - 2e(\Hat{G}[\mathcal{C}])\epsilon\bigg)\\
        & \geq \frac{2}{\mathbf{x}^T\mathbf{x}} \bigg(m - c(k-1)\epsilon \bigg)  > 0   
\end{split}  
\end{equation}
\end{proof}
 \begin{lemma}
 \label{lem11}
 The number of edges in $E(L,R)$ is $|L||R|$, that is, $E(L,R) = \{(l, r) \text{ for all } l\in L, r\in R\}$
 \end{lemma}
 
 \begin{proof}
 If $G[R]$ has no edges, then the result follows from Lemma \ref{lem6}.
 So we assume that there exists exactly one edge, $\{r_1, r_2\}$ in $G[R]$. It also follows from Lemma \ref{lem6} that $e(L, R) \geq |L||R|-2$, and at most two edges $\{l_1, r_1\}$ and $\{l_2, r_2\}$ are missing from $E(L,R)$.
 
 We will first show that at most one edge may be missing. Assume to the contrary that both are missing. Then modifying $G$ to the graph $\overline{G}$ by deleting the edge  $\{r_1, r_2\}$ and adding edges $\{l_1, r_1\}$ and $\{l_2, r_2\}$, strictly increases the spectral radius as may be seen by consequence of Lemma \ref{lem12}.
 However, observe that $\overline{G}$ is $W_{2k+1}$-free. This may be understood by first noting that no vertex in $R$ may be the centre of a $W_{2k+1}$ as this requires the occurrence of a $C_{2k} \subset \overline{G}[L] = G[L]$; and second that no vertex in $L$ may be the centre of a $W_{2k+1}$ as every vertex in $\overline{G}[L]$ is $K_{1, k}$ free and there are no edges in $\overline{G}[R]$. Thus, $\overline{G}[N(l)]$ is $C_t$ free for all $t \geq 2k - 1$ and $l \in L$.
 This contradicts the fact that $\lambda_1(G) \geq \lambda_1(\overline{G})$. So, $e(L, R) \geq |L||R|-1$, and at most one edge $\{l_1, r_1\}$ may be missing from $E(L,R)$. 
 
 Now, the neighborhood of $r_2$ may not contain a $C_{2k}$. Observe that this implies that there are no $C_{2k}$ contained in $G[L \cup \{r_1\}]$, which further implies that there are no $P_{2k-1}$ in $G[L]$ that do not have $l_1$ as one of the end points. It follows that no $P_{2k}$ can be contained in $G[L]$, as we could then select a path on $2k-1$ vertices without $l_1$ as one of its end points, whose vertices along with $r_1, r_2$, induce a graph with $W_{2k+1}$ as a subgraph.
 
 Let $G^+$ be the modification of $G$ obtained by adding the edge $\{l_1, r_1\}$. Then $\lambda_1(G^+) > \lambda_1(G)$, which implies that $G^+ \supset W_{2k+1}$.
 
 Note that either:
 \begin{itemize}
     \item [(i)] for some $l \in L$, the edge $\{l_1, r_1\}$ is part of a $W_{2k+1} \subset G^+[N(l) \cup \{l\}]$, or
     \item[(ii)] $r_1$ is the centre of a $W_{2k+1}$ in $G^+$ with $r_2, l_1 \in C_{2k} \subset G^+[N(r_1)]$, or
     \item[(iii)] the edge $\{l_1, r_1\}$ is part of a $C_{2k} \subset G^+[N(r_2)]$.
 \end{itemize}

In the first case, if $l \in L$ is such that $\{l_1, r_1\}$ is part of a $W_{2k+1} \subset G^+[N(l) \cup \{l\}]$. Then note that the $C_{2k} \subset G^+[N(l)]$ has at most $k-1$ vertices that lie in $L$. Therefore, there are at least $k+1$ vertices of $R$ in the $C_{2k}$. Now observe that other than the edge $\{r_1, r_2\}$, no two vertices of $R$ are adjacent to each other in the $C_{2k}$. Therefore, the number of vertices of $R$ lying in the $C_{2k}$ is maximized when the vertices of the $C_{2k}$ alternate between the left, $L$, and right, $R$, parts as we go along a path on $2k$ vertices in the $C_{2k}$, starting at $r_2$ and ending at $r_1$. Therefore. there are at most $k$ vertices from $R$ in the $C_{2k}$. This is a contradiction. Hence the first case is not possible. 
 
 In the second case, if $d_L(l_1) \geq 2$, then either there exists a $P_{2k}$ in $G^+[L] = G[L]$, or the vertices of the $C_{2k}$ in $G^+$ induce a graph with another $C_{2k} \ni r_1$ in $G$. The vertices of this new cycle are all adjacent to $r_2$, and hence we can say that $G$ would contain a $W_{2k+1}$ if $d_L(l_1) \geq 2$ under the first case.

Similarly, if the third case were to be true and $d_L(l_1) \geq 2$, then either there exists a $P_{2k}$ in $G^+[L] = G[L]$, or the vertices of the $C_{2k}$ induce a graph with another $C_{2k}$ as a subgraph. All of these vertices are adjacent to  $r_2$, and hence again, we would already have had a $W_{2k+1}$ in $G$.
 
Next, let us consider the situation when $d_L(l_1) = 1$. Let $\mathcal{C}$ be the connected component containing $l_1$ in $G^+[L]$. Now $|\mathcal{C}| \geq 2k-1$ and has no $P_{2k}$. If $|\mathcal{C}| \geq 2k$, then we can further modify the graph $G^+$ to $\Hat{G}$,  such that $\Hat{G}[\mathcal{C}]$ has no $P_{2k-1}$ and every vertex in $\mathcal{C}$ has internal degree $k-1$ with at most one vertex having internal degree $k-2$. This implies that $\Hat{G}$ has at least $\frac{|\mathcal{C}|(k-1) - 1}{2} - \frac{|\mathcal{C}|(k-1) - (k-2)}{2} = \frac{k-3}{2}$ more edges than $G^+$, and at least $\frac{k-1}{2}$ edges more than $G$. Also, note that the following loose upper bound $|\mathcal{C}| < k^k$ holds since $\mathcal{C}$ has no $P_{2k-1}$ as a subgraph.
 
 Therefore, using Lemma \ref{lem12} with $m = \frac{k-1}{2}$, $c = k^k$ and
$\epsilon < \frac{1}{2k^k}$,  $\lambda_1(\Hat{G}) - \lambda_1(G)>0$ for $n$ large enough. 
Hence,  $\lambda_1(\Hat{G}) > \lambda_1(G)$. However $G[L]$ has no $P_{2k-1}$ or $K_{1,k}$ as subgraphs, and hence $G$ has no $W_{2k+1}$ as a subgraph, which is a contradiction. Therefore, $G$ must contain the edge $\{l_1, r_1\}$.
 
 Finally, if $|\mathcal{C}| = 2k-1$, consider the following. Let $\mathcal{D}$ be a different connected component in $G^+[L]$. Then $|\mathcal{C} \cup \mathcal{D}| \geq 2k$. Let $\mathcal{E}:= \mathcal{C} \cup \mathcal{D} \ni l_1$.  Now similarly modify $G^+$ to $\Tilde{G}$ such that $\Tilde{G}[\mathcal{E}]$ has no $P_{2k-1}$ and every vertex in $\mathcal{E}$ has internal degree $k-1$ with at most one vertex having internal degree $k-2$. This again implies that $\Tilde{G}$ has at least $\frac{|\mathcal{E}|(k-1) - 1}{2} - \frac{|\mathcal{E}|(k-1) - (k-2)}{2} = \frac{k-3}{2}$ more edges than $G^+$, and at least $\frac{k-1}{2}$ edges more than $G$. Here too $|\mathcal{E}| < k^k$ holds since $D$ has no $P_{2k-1}$ as a subgraph.
 
 The rest of the proof follows the same arguments as above. Therefore, $E(L,R)$ must contain the edge $\{l_1, r_1\}$, and $e(L,R) = |L||R|$. 
 
 \end{proof}

 \begin{lemma}
If $(k-1)|L|$ is even then $G[L]$ is a $(k-1)$-regular graph and otherwise $G[L]$ is a $(k-1)$-nearly regular graph. Furthermore, $e(G[R]) = 1$.
\end{lemma}
\begin{proof}
By way of contradiction, assume that the statement is not true. Thus, either $G[L]$ has at least two vertices of degree not more than $k-2$ or at least one vertex of degree at most $k-3$, or $G[R]$ has no edge. Let $E_2$ be a set of edges such that $G \cup E_2$ is a (potentially not simple) graph that induces one edge on $R$ and where the graph induced by $L$ is $(k-1)$-regular. By the assumption, the sum of the entries in $E_2$ is at least $2$.  
Now, by Lemma \ref{lem5.6},
\begin{align*}
\lambda_1(G) = \frac{\mathbf{x}^TA(G) \mathbf{x}}{\mathbf{x}^T\mathbf{x}} &\leq \mu -  \frac{\mathbf{x}^TA(E_2) \mathbf{x}}{\mathbf{x}^T\mathbf{x}} + \frac{\mathbf{x}^TA(G[R]) \mathbf{x}}{\mathbf{x}^T\mathbf{x}} \\
&< \mu - \frac{2(1-\epsilon)^2}{\mathbf{x}^T\mathbf{x}} + \frac{2(1+\epsilon)^2}{\mathbf{x}^T\mathbf{x}}\\
&< \mu + \frac{8\epsilon}{(1-\epsilon)^2n}.
\end{align*}
This is a contradiction to \eqref{lower bound G} for $\epsilon$ small enough.
\end{proof}

Finally, we show that $|L|$ and $|R|$ differ by at most $2$. This completes the proof of Theorem \ref{thm4}.
 
 \begin{lemma}
\label{almost equipartite}
 For $G \in \textup{SPEX}(n, W_{2k+1})$, with maximum cut $L, R$, we have \[\frac{n}{2} - 1 \leq |L|, |R| \leq \frac{n}{2} + 1.\] 
 \end{lemma}
 
\begin{proof}
Let  $|L| = \frac{n}{2} + s$ and $|R| = \frac{n}{2} - s$. We will show $|s|\leq 1$. Let  $B = \begin{bmatrix} 
k-1 & \frac{n}{2} - s\\
\frac{n}{2} + s & 0\\
\end{bmatrix}$. We know that $G$ is a complete bipartite graph with maximum cut $L, R$ and $e(G[R]) \leq 1$ and $G[L]$ a $k-1$ regular or nearly graph. Hence, $\lambda_1(G) \leq \lambda_1(B) + \frac{2}{\mathbf{x}^T\mathbf{x}}$ by Lemma \ref{thm5}. Combining this with \eqref{lower bound G} gives 

\begin{align*}
\frac{k-1 + \sqrt{(k-1)^2 + n^2-1}}{2} &< \frac{k-1 + \sqrt{(k-1)^2 + (n^2-4s^2)}}{2} + \frac{2}{\mathbf{x}^T\mathbf{x}} \\& \leq \frac{k-1 + \sqrt{(k-1)^2 + (n^2-4s^2)}}{2} + \frac{2}{(1-\epsilon)^2n}.
\end{align*}

Simplifying shows that $|s|\leq 1$ for $n$ large enough.

\end{proof}
\section{The sizes of $|L|$ and $|R|$ in Theorem \ref{thm4}}\label{section partition sizes}
In this section we will show that when $k$ is odd, then $|L| = \ceil{\frac{n}{2}}$ and when $k $ is even, then $|L|$ is constrained as follows, depending on the value of $n \pmod 4$.
\begin{enumerate}
    \item[(i)] For $n \equiv 0 \pmod{4}$, $|L| = \frac{n}{2}$;
    \item[(ii)] For $n \equiv 1 \pmod{4}$, $|L| = \floor{\frac{n}{2}}$;
    \item[(iii)] For $n \equiv 2 \pmod{4}$, $|L| \in \{\frac{n}{2}, \frac{n}{2} + 1\}$;
    \item[(iv)] For $n \equiv 3 \pmod{4}$, $|L| = \ceil{\frac{n}{2}}$.
\end{enumerate}
We will now try to fine-tune our argument a bit to show (i), (iii), and (iv), by using quotient graphs with three parts instead of just two parts, as done so far.

Let \[\Pi_s = \begin{bmatrix}
k-1 & \frac{n}{2} - s - 2 & 2 \\
\frac{n}{2} + s & 0 & 0 \\
\frac{n}{2} + s & 0 & 1 
\end{bmatrix}\] and $P_s(\lambda) = \lambda^3 - k \lambda^2 - \bigg(\frac{n^2}{4} - s^2 - k + 1 \bigg) \lambda + \frac{n^2}{4} - s^2 - n - 2s$, be its characteristic polynomial. Then for $\alpha \in \{0.5, 1\}$, we have $P_{\alpha}(\lambda) - P_{- \alpha}(\lambda) = -4\alpha$; and $P_0(\lambda) - P_1(\lambda) = 3 - \lambda < 0$, for $\lambda$ near $\lambda_1(\Pi_0)$. Since the coefficient of $\lambda^3$ in $P_s$ is positive, all three roots of $P_s$ are real and the largest root is simple, it implies that $\lambda_1(\Pi_{\alpha}) > \lambda_1(\Pi_{-\alpha})$ and $\lambda_1(\Pi_0) > \lambda_1(\Pi_1)$.

Let $\mathcal{G}_s$ be the family of graphs that consist of a complete bipartite graph with parts of size $\frac{n}{2} + s$ and $\frac{n}{2} - s$ along with a graph from $\mathcal{U}_{k, n/2+s}$ embedded in the first part with $\frac{n}{2}+s$ vertices and a single edge embedded in the other part. For any arbitrary graph $G_s \in \mathcal{G}_s$, we can say that $\lambda_1(G_s) \leq \lambda_1(\Pi_s)$ (as $\Pi_s$ is a quotient matrix for the adjacency matrix of $G_s$ with at most one loop added to make every vertex in the first part have same degree) with equality if and only if $(k-1)(\frac{n}{2} + s)$ is even. We know by Theorem \ref{thm4} that $G \in \mathcal{G}_s$ for some $s$ satisfying $-1 \leq s \leq 1$. Then, the previous paragraph implies that $s \in \{0, 0.5\}$ (i.e. $|L| = \ceil{\frac{n}{2}}$), whenever $k$ is odd or $\ceil{\frac{n}{2}}$ is even (implying (i) and (iv)). Further, $s \in \{0, 1\}$ (i.e. $|L| = \frac{n}{2}$ or $\frac{n}{2}+1$), if $\frac{n}{2}$ is odd ( implying (iii)).

 When $k$ is even and $n\equiv 1,2 \pmod{4}$ we argue similarly to Lemma \ref{lemma first lower bound quotient matrix}. Let 
 \[
 Q_s = \begin{bmatrix}
 k-1 & \frac{n}{2}-s\\
 \frac{n}{2} +s & 0
 \end{bmatrix}.
 \]
 Let $uv$ be the edge that is embedded in $R$ and let $u_\ell$ be the vertex that has internal degree $k-2$ in the case that $(\frac{n}{2}+s)(k-1)$ is odd. Let $\mu_s$ be the spectral radius of $Q_s$ with eigenvector $\begin{bmatrix}
 1 & \eta
 \end{bmatrix}^T$ and let $\mathbf{z}$ be the $n$-dimensional vector where the first $\frac{n}{2}+s$ entries are $1$ and the last $\frac{n}{2}-s$ entries are $\eta$. Note that $\eta = 1-o(1)$ as $n\to \infty$. Let $E$ be the adjacency matrix of the edge $uv$ and (if $u_\ell$ exists) the loop $u_\ell$ with weight $-1$. That is, $E$ is a matrix with exactly two entries equal to $1$ and if $(n/2+s)(k-1)$ is odd a single diagonal entry equal to $-1$. Then 
 \[
 \frac{\mathbf{z}^T(A(G)-E)\mathbf{z}}{\mathbf{z}^T\mathbf{z}} +  \frac{\mathbf{z}^TE\mathbf{z}}{\mathbf{z}^T\mathbf{z}} =  \frac{\mathbf{z}^T(A(G))\mathbf{z}}{\mathbf{z}^T\mathbf{z}}\leq\lambda_1  = \frac{\mathbf{x}^T(A(G))\mathbf{x}}{\mathbf{x}^T\mathbf{x}} =  \frac{\mathbf{x}^T(A(G)-E)\mathbf{x}}{\mathbf{x}^T\mathbf{x}} +   \frac{\mathbf{x}^TE\mathbf{x}}{\mathbf{x}^T\mathbf{x}}.
 \]
 Since $Q_s$ is the quotient matrix of an equitable partition of the graph $G$ minus edge $uv$ and plus (if $u_\ell$ exists) loop $u_\ell$, and since $\mathbf{z}$ is an eigenvector for $A(G) - E$, we have 
 \begin{equation}\label{final upper and lower bounds}
 \mu_s + \frac{\mathbf{z}^TE\mathbf{z}}{\mathbf{z}^T\mathbf{z}} \leq \lambda_1 \leq \mu_s + \frac{\mathbf{x}^T E\mathbf{x}}{\mathbf{x}^T\mathbf{x}}
 \end{equation}
 
 \begin{proposition}
 When $n\equiv 1 \pmod{4}$ and $k$ is even, then $s=-1/2$.
 \end{proposition}
 \begin{proof}
 If $n\equiv 1 \pmod{4}$ then the vertex $u_\ell$ exists if $s=1/2$ and does not exist if $s=-1/2$. Let $G_{1/2}$ and $G_{-1/2}$ be arbitrary graphs in $\mathcal{G}_s$ in the cases that $s=1/2$ and $s=-1/2$ respectively. Then by \eqref{final upper and lower bounds}, we have 
 \[
 \lambda_1(G_{-1/2}) \geq \mu_{-1/2} + \frac{2\eta^2}{\mathbf{z}^T\mathbf{z}} = \mu_{-1/2} + \frac{2-o(1)}{n}.
 \]
 On the other hand, if $G_{1/2}$ were extremal we would have by \eqref{final upper and lower bounds} and Lemma \ref{lem5.6} that 
 \[
 \lambda_1(G_{1/2}) \leq \mu_{1/2} + \frac{2\mathbf{x}_u \mathbf{x}_v  - \mathbf{x_{u_\ell}}^2}{\mathbf{x}^T\mathbf{x}} = \mu_{1/2} + \frac{1+o(1)}{n}.
 \]
 Noting that $\mu_{-1/2} = \mu_{1/2}$ completes the proof.
 
 \end{proof}
 
 \begin{proposition}
 When $n\equiv{2} \pmod{4}$ we have $\mathrm{spex}(n, W_{2k+1}) = \frac{k-1+\sqrt{(k-1)^2+n^2}}{2} + \frac{1+o(1)}{n}$.
 \end{proposition}
 
 \begin{proof}
 Let $G_1$ be the graph when $s=1$ and $G_{0}$ be the graph when $s=0$. Note that the vertex $u_\ell$ exists in $G_0$ and does not exist in $G_1$. Hence if $G_1$ is extremal, by \eqref{final upper and lower bounds} we have
 \[
 \lambda_1(G_1) = \mu_1 + \frac{2-o(1)}{n},
 \]
 and if $G_0$ is extremal we have 
 \[
 \lambda_1(G_0) = \mu_0 + \frac{1+o(1)}{n}.
 \]
 Comparing these two quantities shows that they differ by $o(1/n)$.
 \end{proof}
 
 Finally we notice that if $n\not\equiv 2 \pmod{4}$ or if $k$ is odd, then the extremal graph with the appropriately chosen $s$ has an equitable partition with $\Pi_s$ as the quotient matrix, and hence $\lambda_1(G)$ is equal to the largest root of $P_s(\lambda)$ in these cases.

\section{Conclusion}\label{section conclusion}
In this paper, we determined the structure of the graphs in $\mathrm{SPEX}(n, W_{2k+1})$ for all $k\not\in \{4,5\}$ and for $n$ large enough. We believe that the extremal graphs when $k\in \{4,5\}$ have the same structure, and it would be interesting to prove this. The main technical hurdle is to prove Lemma \ref{lem9} when $k\in \{4,5\}$. We note that it is a bit delicate. For example, for $k=4$, the join of a disjoint union of triangles with a matching has spectral radius very close to a complete bipartite graph with disjoint copies of $K_4$ in one side and a single edge in the other side.

It would also be interesting to determine how large $n$ needs to be as a function of $k$ for our theorems to hold. Our use of the Triangle Removal Lemma means that our ``sufficiently large $n$" is likely much larger than it needs to be.

Finally, we end with a general conjecture.

\begin{conjecture}
Let $F$ be any graph such that the graphs in $\mathrm{EX}(n, F)$ are Tur\'an graphs plus $O(1)$ edges. Then $\mathrm{SPEX}(n, F) \subset \mathrm{EX}(n, F)$ for $n$ large enough.
\end{conjecture}

Nikiforov's result \cite{Nsat} shows that this is true for critical graphs, when the $O(1)$ is replaced by $0$. We believe that similar methods to what is in this paper and in \cite{FG} would help to prove the conjecture for any fixed graph $F$ satisfying the hypotheses.

	\bibliographystyle{plain}
	\bibliography{bib.bib}

\end{document}